\documentclass[12pt]{article}

\usepackage{amsmath}
\usepackage{amsfonts}
\usepackage{latexsym}
\usepackage{amssymb}
\usepackage{curves}
\usepackage{graphicx}
\usepackage{amsthm,amscd, amssymb}

\newtheorem{thm}{Theorem}[section]
\newtheorem{pro}[thm]{Proposition}
\newtheorem{cor}[thm]{Corollary}
\newtheorem{lem}[thm]{Lemma}

\newcommand{\noin}{\noindent}
\newcommand{\dozspace}{\;\;\;\;\;\;\;\;\;\;\;\;}
\newcommand{\dozback}{\!\!\!\!\!\!\!\!\!\!\!\!}

\newcommand{\la}{\langle}
\newcommand{\ra}{\rangle}

\newcommand{\bbb}{{\mbox{}}}
\newcommand{\ppp}{{\vphantom{\gamma}}}

\newcommand{\BB}{{\mathbb{B}}}

\newcommand{\FF}{{\mathbb{F}}}

\newcommand{\KK}{{\mathbb{K}}}

\newcommand{\QQ}{{\mathbb{Q}}}

\newcommand{\ZZ}{{\mathbb{Z}}}

\newcommand{\Con}{\mbox{\rm{Con}}}

\newcommand{\Def}{\mbox{\rm{Def}}}

\newcommand{\Ind}{\mbox{\rm{Ind}}}
\newcommand{\Inf}{\mbox{\rm{Inf}}}

\newcommand{\Res}{\mbox{\rm{Res}}}

\newcommand{\ind}{\mbox{\rm{ind}}}

\newcommand{\lin}{\mbox{\rm{lin}}}

\newcommand{\tr}{\mbox{\rm{tr}}}

\newcommand{\oo}{\overline}

\renewcommand{\tt}{\widetilde}
\newcommand{\hh}{\widehat}

\newcommand{\cC}{{\cal C}}

\newcommand{\cE}{{\cal E}}

\newcommand{\cL}{{\cal L}}
\newcommand{\cM}{{\cal M}}
\newcommand{\cN}{{\cal N}}

\newcommand{\cP}{{\cal P}}

\newcommand{\cR}{{\cal R}}
\newcommand{\cS}{{\cal S}}

\begin{document}


\title{An inversion formula for the primitive \\
idempotents of the trivial source algebra}

\author{\large Laurence Barker \\ \mbox{} \\
\normalsize Department of Mathematics \\
\normalsize Bilkent University \\
\normalsize 06800 Bilkent, Ankara \\
\normalsize Turkey}

\maketitle

\small

\begin{abstract}
\noin Formulas for the primitive
idempotents of the trivial source algebra, in
characteristic zero, have been given by Boltje and
Bouc--Th\'{e}venaz. We shall give another formula
for those idempotents, expressing them as
linear combinations of the elements of a canonical
basis for the integral ring. The formula is an inversion
formula analogous to the Gluck--Yoshida formula
for the primitive idempotents of the Burnside algebra.
It involves all the irreducible characters of all the
normalizers of $p$-subgroups. As a corollary,
we shall show that the linearization map from the
monomial Burnside ring has a matrix whose entries
can be expressed in terms of the above Brauer
characters and some reduced Euler characteristics
of posets.

\smallskip
\noin 2010 {\it Mathematics Subject Classification.}
Primary: 20C20; Secondary: 19A22.

\smallskip
\noin {\it Keywords:} Trivial source ring, $p$-permutation
ring, table of marks, monimial Burnside ring.

\end{abstract}

\let\thefootnote\relax\footnote{This work was
supported by T\"{u}bitak Scientific and Technological
Research Funding Program 1001 under grant number
114F078.}

\section{Introduction}

The Burnside ring $B(G)$ of a finite group $G$ has
a basis consisting of the isomorphism classes of
transitive $G$-sets. Extending to coefficients in
$\QQ$, the isomorphism classes of transitive
$G$-sets still comprise a basis for $\QQ B(G)$.
The primitive idempotents of $\QQ B(G)$ comprise
another basis for $\QQ B(G)$. It is straightforward
to express each element of the former basis as a
linear combination of elements of the latter basis.
The matrix associated with those linear
combinations was studied by Burnside, who
called it the table of marks for $G$. Gluck
\cite{Glu81} and Yoshida \cite[Section 3]{Yos83}
inverted the table of marks, expressing each
primitive idempotent as a linear combination
of isomorphism classes of transitive $G$-sets.
Boltje \cite[Section 3]{Bol} gave similar inversion
formulas for the primitive idempotents of various
monomial Burnside rings and character rings.

Let $\FF$ be a field with prime characteritsic $p$
and let $\KK$ be a field of characteristic zero.
Throughout, we shall assume that $\KK$ and $\FF$
are sufficiently large in the sense that they own
enough $p'$-th roots of unity. Fixing an
isomorphism between a sufficiently large group
of $p'$-th roots of unity in $\KK$ and a group of
$p'$-th roots of unity in $\FF$, we can understand
that, for modules of all the group algebras over
$\FF$ that come into consideration, the Brauer
characters have values in $\KK$.

Recall, given a finite-dimensional $\FF G$-module
$M$, then every indecomposable direct summand
of $M$ has a trivial source if and only if $M$ has a
basis stabilized by a Sylow $p$-subgroup of $G$.
When those equivalent conditions hold, we call
$M$ a {\bf trivial source} $\FF G$-module (or
a {\bf $p$-permutation} $\FF G$-module). The
{\bf trivial source ring} $T_\FF(G)$ of $\FF G$,
(also called the {\bf $p$-permutation ring} of
$\FF G$), is defined to be the abelian group
generated by the isomorphism classes $[M]$
of trivial source $\FF G$-modules $M$, subject
to the relations $[M] + [M'] = [M \oplus M']$. We
make $T_\FF(G)$ become a ring with multiplication
coming from tensor product, $[M][M'] = [M
\otimes_\FF M']$. For an introduction to the theory
of trivial source modules and the trivial source
ring, see Section 2 or, for more detail,
Bouc--Th\'{e}venaz \cite[Section 2]{BT10}.

Let $P$ be a $p$-subgroup of $G$. Write
$\oo{N}_G(P) = N_G(P)/P$. Let $\phi$ be an
irreducible Brauer character of $\FF \oo{N}_G(P)$.
Write $E_\phi$ for the indecomposable projective
$\FF \oo{N}_G(P)$-module such that the simple
head $E_\phi / J(E_\phi)$ has Brauer character
$\phi$. We define an induced and inflated
$\FF G$-module
$$N_{P, \phi} = \bbb_G^\ppp \Ind
  \bbb_{N(P)}^\ppp \Inf
  \bbb_{\oo{N}(P)}(E_\phi) \; .$$
In Section 2 we shall show that, as $P$ and
$\phi$ vary, the isomorphism classes having
the form $[N_{P, \phi}]$ comprise a $\ZZ$-basis
for $T_\FF(G)$ and a $\KK$-basis for
$\KK T_\FF(G)$. We shall call this basis the
{\bf canonical basis} for $T_\FF(G)$. Also in
Section 2, we shall also give a citation
for the well-known fact that the primitive
idempotents of $\KK T_\FF(G)$ comprise a
$\KK$-basis for $\KK T_\FF(G)$.

Formulas for the primitive idempotents of
$\KK T_\FF(G)$ were given by Boltje
\cite[3.6]{Bol} and Bouc--Th\'{e}venaz
\cite[4.12]{BT10}. Our main result,
Theorem \ref{2.4}, asserts a different
formula which expresses each primitive
idempotent as a linear combination
of the elements of the canonical basis. Our
formula, involving irreducible Brauer characters
of $\FF \oo{N}_G(P)$, is an inversion
formula in the sense that it is analogous to
the inversion formulas of Gluck and Yoshida.
We shall prove it, in Section 3, by considering
a matrix analogous to the table of marks.

The monomial Burnside ring $B_\FF(G)$ has
a $\ZZ$-basis consisting of the isomorphism
classes of transitive $F^\times$-fibred $G$-sets,
where $F^\times = \FF - \{ 0 \}$. Theorem
\ref{4.1} gives a formula for the matrix
representing the linearization map
$\lin_G : B_\FF(G) \rightarrow T_\FF(G)$ with
respect to the above basis of $B_\FF(G)$ and
the canonical basis for $T_\FF(G)$. The formula
again involves irreducible Brauer characters of
$\FF \oo{N}_G(P)$.

\section{Statement of the idempotent formula}

After briefly reviewing some well-known material
on the primitive idempotents of the trivial source
algebra $\KK T_\FF(P)$, we shall state a formula
for those idempotents and deduce a corollary.

Let $P$ be a $p$-subgroup of $G$ and let $M$
be a trivial source $\FF G$-module. We write
$M^P$ to denote the subspace of $P$-fixed
elements of $M$. For $Q \leq P$, we define the
relative trace map to be the linear map
$$\tr_Q^P \: : \: M^Q \ni m \mapsto
  \sum_{uQ \subseteq P} {}^u m \in M^P \; .$$
We define the {\bf Brauer quotient} of $M$
to be the $\FF \oo{N}_G(P)$-module
$$M[P] = M^P /
  \sum_{Q < P} \tr_Q^P(M^Q) \; .$$
Since $M$ is a trivial source module, it has
a $P$-stable basis $\Omega$. Writing
$\Omega^P$ for the set of $P$-fixed elements
of $\Omega$, then $M[P]$ can be identified
with the $\FF$-span $\FF \Omega^P$ of
$\Omega^P$.

Let $\cC(G)$ be the set of pairs $(P, \phi)$ where
$P$ is a $p$-subgroup of $G$ and $\phi$ is an
irreducible Brauer character of $\FF \oo{N}_G(P)$.
We allow $G$ to act on $\cC(G)$ such that
${}^g (P, \phi) = ({}^g P, {}^g \phi)$ for $g \in G$.
For each $(P, \phi) \in \cC(G)$, we define
$M_{P, \phi}^G$ to be the indecomposable
$\FF G$-module, unique up to isomorphism, such
that $M_{P, \phi}^G$ has vertex $P$ and
$M_{P, \phi}^G$ is in Green correspondence with
the inflated $\FF N_G(P)$-module $\bbb_{N(P)}^\ppp
\Inf \bbb_{\oo{N}(P)}^\ppp (E_\phi)$. Plainly,
$M_{P, \phi}^G$ is a trivial source
$\FF G$-module and the isomorphism class
$[M_{P, \phi}^G]$ depends only on the $G$-orbit
of $(P, \phi)$. Using Bouc--Th\'{e}venaz
\cite[2.7]{BT10}, it is straightforward to show that
$M_{P, \phi}^G[P] \cong E_\phi$.

\begin{pro} \label{2.1}
{\rm (See \cite[2.9]{BT10}.)} As $(P, \phi)$ runs over
representatives of the $G$-orbits of $\cC(G)$, the
isomorphism classes $[M_{P, \phi}^G]$ comprise a
$\ZZ$-basis for $T_\FF(G)$.
\end{pro}

Let $\cE(G)$ denote the set of pairs $(Q, [s])$
where $Q$ is a $p$-subgroup of $G$ and $[s]$
is the conjugacy class of a $p'$-element of
$\oo{N}_G(Q)$. We allow $G$ to act on $\cE(G)$
such that ${}^g(Q, [s]) = ({}^g Q, {}^g [s])$.
For each $(Q, [s]) \in \cE(G)$, we
define a linear map
$$\epsilon_{Q, s}^G \: : \: \KK T_\FF(G)
  \rightarrow \KK$$
such that, given a trivial source $\FF G$-module
$M$, then $\epsilon_{Q, s}^G[M]$ is the value,
at $s$, of the Brauer character of $M[Q]$. It is
easy to show that $\epsilon_{Q, s}^G$ depends
only on the $G$-orbit of $(Q, [s])$. Part of
\cite[2.18]{BT10} asserts that $\epsilon_{Q, s}^G$
is an algebra map.

\begin{pro} \label{2.2}
{\rm (See \cite[2.18, 2.19]{BT10}.)} For each
$G$-orbit of elements $(Q, [s]) \in \cE(G)$, there
exists a unique primitive idempotent $e_{Q, s}^G$
of $\KK T_\FF(G)$ such that $\epsilon_{Q, s}^G
(e_{Q, s}^G) = 1$. Furthermore,
$$\KK T_\FF(G) = \bigoplus_{(Q, [s]) \in_G \cE(G)}
  \KK \, e_{Q, s}^G$$
as a direct sum of algebras isomorphic to $\KK$,
the notation indicating that $(Q, [s])$ runs over
representatives of the $G$-orbits of $\cE(G)$.
\end{pro}

By the two propositions above, we can write
$$[M_{P, \phi}^G] = \sum_{(Q, [s]) \in_G \cE(G)}
  \!\!\!\! m_G(Q, s; P, \phi) \, e_{Q, s}^G \; , \;\;\;\;\;\;
  e_{Q, s}^G = \sum_{(P, \phi) \in_G \cC(G)}
  \!\!\!\! m_G^{-1}(P, \phi; Q, s) \, [M_{P, \phi}^G] \; .$$
Some partial results about the coefficients
$m_G(Q, s; P, \phi), m_G^{-1}(P, \phi; Q, s) \in \KK$
will be given in the next section, but we do not have
an explicit general formula for them. Of course,
$$m_G(Q, s; P, \phi) =
  \epsilon_{Q, s}^G [M_{P, \phi}^G]$$
but this equality is not explicit enough to admit
much manipulation. That is why we
shall focus on a different $\ZZ$-basis for $T_\FF(G)$.
Observe that the isomorphism class $[N_{P, \phi}^G]$
depends only on the $G$-orbit of $(P, \phi)$.

\begin{pro} \label{2.3}
As $(P, \phi)$ runs over representatives of the
$G$-orbits of $\cC(G)$, the isomorphism classes
$[N_{P, \phi}^G]$ comprise a $\ZZ$-basis for $T_\FF(G)$.
\end{pro}

\begin{proof}
By the Green correspondence, $N_{P, \phi}^G$ is the
direct sum of $M_{P, \phi}^G$ and an $\FF G$-module
whose indecomposable direct summands all have
vertices strictly contained in $P$. The assertion now
follows from Proposition \ref{2.1}.
\end{proof}

By Propositions \ref{2.2} and \ref{2.3}, we can
write
$$[N_{P, \phi}^G] = \sum_{(Q, [s]) \in_G \cE(G)}
  \!\!\!\! n_G(Q, s; P, \phi) \, e_{Q, s}^G \; , \;\;\;\;\;\;
  e_{Q, s}^G = \sum_{(P, \phi) \in_G \cC(G)}
  \!\!\!\! n_G^{-1}(P, \phi; Q, s) \, [N_{P, \phi}^G]$$
with $n_G(Q, s; P, \phi), n_G^{-1}(P, \phi; Q, s) \in
\KK$. Observe that
$$n_G(Q, s; P, \phi) =
  \epsilon_{Q, s}^G [N_{P, \phi}^G] \; .$$
We shall give an explicit formula for $n_G(Q, s;
P, \phi)$ in Proposition \ref{3.4} and an explicit
formula for $n_G^{-1}(P, \phi; Q, s)$ in Theorem
\ref{3.6}. As we shall prove in Section 3, the
latter formula is equivalent to the next theorem.

To state the theorem, we need some more
notation. Recall that any finite poset $\cP$ can
be associated with a simplicial complex whose
simplices are the chains in $\cP$. We define
the {\bf reduced Euler characteristic} of $\cP$,
denoted $\tt{\chi}(\cP)$, to be the reduced
Euler characteristic of the simplicial complex.
The {\bf M\"{o}bius function} $\mu : \cP \times
\cP \rightarrow \ZZ$ is defined to be the function
such that, given $x, y \in \cP$, then $\mu(x, x)
= 1$; if $x \not\leq y$ then $\mu(x, y) = 0$; if
$x < y$, writing $(x, y)_\cP = \{ z \in P : x
< z < y \}$, then $\mu(x, y) = \tt{\chi}((x, y)_\cP)$.
It is well-known that, for $x < y$, we have a
recurrence relation $\sum_z \mu(x, z) = 0 =
\sum_z (z, y)$ summed over $z$ such that
$x \leq z \leq y$. We shall be making use of
that recurrence relation in the next section.

Let $\cS_p(G)$ denote the $G$-poset
of $p$-subgroups of $G$. For $g \in G$, the
$\la g \ra$-fixed subposet $\cS_p(G)^{\la g \ra}$
consists of those $p$-subgroups $R$ such that
$g \in N_G(R)$. We write $\mu_g$ to denote the
M\"{o}bius function of $\cS_p(G)^{\la g \ra}$.

\begin{thm} \label{2.4}
Given $(Q, [s]) \in \cE(G)$, then
$$e_{Q, s}^G = \frac{1}{|N_G(Q)|} \sum_{(P,
  \phi) \in \cC(G), \, gP \in \oo{N}_G(P)_{p'} \: : \:
  P \leq Q, \, g \in N_G(Q), \, gQ \in [s]} \dozback
  \dozback |P| \, \phi(g^{-1} P) \, \mu_g(P, Q) \,
  [N_{P, \phi}^G]$$
where $\oo{N}_G(P)_{p'}$ denotes the set of
$p'$-elements of $\oo{N}_G(P)$.
\end{thm}

The formula makes sense because the conditions
$P \leq Q$ and $g \in N_G(Q)$ imply that the
$p'$-element $gP \in \oo{N}_G(P)_{p'}$
determines the $p'$-element $gQ \in
\oo{N}_G(Q)_{p'}$.

\begin{lem} \label{2.5}
Given $p$-subgroups $P, Q \leq G$ and $g \in
N_G(P, Q)$ such that $\mu_g(P, Q) \neq 0$, then
$\Phi(Q) \leq P \leq Q$ where $\Phi(Q)$ denotes
the Frattini subgroup of $Q$. So the sum in
Theorem \ref{2.4} can be taken over only those
indices such that $\Phi(Q) \leq P \leq Q$.
\end{lem}

\begin{proof}
We apply the technique of conical contraction
of posets, described in Benson
\cite[6.4.5, 6.4.6]{Ben91}. The condition
$\mu_g(P, Q) \neq 0$ plainly implies that
$P \leq Q$. For a contradiction, suppose
that $\Phi(Q) \not\leq P$. Let $\cR$ be the
open interval $(P, Q)$ in $\cS_p(G)^{\la g \ra}$.
Then $\cR$ admits a contraction given by
$R \mapsto R \Phi(Q) \mapsto P \Phi(Q)$ for
$R \in \cR$. This contradicts the condition
that $\mu_g(P, Q) \neq 0$.
\end{proof}

\begin{cor} \label{2.6}
Let $(Q, [s]) \in \cE(G)$ and let
$$I = \sum_{P \: : \: \Phi(Q) \leq P \leq Q}
  \bbb_G^\ppp \Ind \bbb_{N_G(P)}^\ppp
  (\KK T_\FF(N_G(P))) \; .$$
Then $I$ is an ideal of $\KK T_\FF(G)$ and
$e_{Q, s}^G \in I$.
\end{cor}

\begin{proof}
By the Frobenius relations, $I$ is an ideal.
By Lemma \ref{2.5}, $e_{Q, s}^G \in I$.
\end{proof}

\section{Proof of the idempotent formula}

Let $\cM$ be the matrix such that the rows are
indexed by the $G$-orbits of elements $(Q, [s])
\in \cE(G)$, the columns are indexed by the
$G$-orbits of elements $(P, \phi) \in \cC(G)$
and the $((Q, [s]), (P, \phi))$ entry is
$m_G(Q, s; P, \phi)$. Let $\cN$ be the matrix
with the same indexing of rows and columns
and with $((Q, [s]), (P, \phi))$ entry
$n_G(Q, s; P, \phi)$. We shall be making a
study of the matrices $\cM$, $\cM^{-1}$,
$\cN$, $\cN^{-1}$.

\begin{lem} \label{3.1}
Let $(P, \phi) \in \cC(G)$ and $(Q, [s]) \in \cE(G)$.

\noin {\bf (1)} If $Q \not\leq_G P$, then
$m_G(Q, s; P, \phi) = n_G(Q, s; P, \phi) = 0$.

\noin {\bf (2)} If $P \not\leq_G Q$, then
$m_G^{-1}(P, \phi; Q, s) =
n_G^{-1}(P, \phi; Q, s) = 0$.
\end{lem}

\begin{proof}
The $\FF G$-module $M_{P, \phi}^G$ is a
direct summand of $N_{P, \phi}^G$ which is
a direct summand of the permutation
$\FF G$-module $\FF G/P$. If $Q \not\leq_G
P$, then $\FF G/P[Q] = 0$, hence
$M_{P, \phi}^G[Q] = N_{P, \phi}^G[Q] = 0$.
Part (1) follows. Thus, listing the conjugacy
classes of $p$-subgroups of $G$ in
non-decreasing order, and regarding $\cM$
and $\cN$ as arrays of submatrices with rows
indexed by the $p$-subgroup $Q$ and
columns indexed by the $p$-subgroup $P$,
then $\cM$ and $\cN$ have an upper
triangular form. The inverses $\cM^{-1}$
and $\cN^{-1}$ have the same upper triangular
form, hence part (2).
\end{proof}

Let $\hh{\phi}$ denote the Brauer character
of the indecomposable projective
$\FF \oo{N}_G(P)$-module $E_\phi$.

\begin{pro} \label{3.2}
Given a $p$-subgroup $P$ of $G$, an
irreducible Brauer character $\phi$ of
$\FF \oo{N}_G(P)$ and a $p'$-element $s$
of $\oo{N}_G(P)$, then

\noin {\bf (1)} $m_G(P, s; P, \phi) =
n_G(P, s; P, \phi) = \hh{\phi}(s)$,

\noin {\bf (2)} $m_G^{-1}(P, \phi; P, s)
= n_G^{-1}(P, \phi; P, s) = \phi(s^{-1})
/ |C_{\oo{N}_G(P)}(s)|$.
\end{pro}

\begin{proof}
Part (1) holds because $M_{P, \phi}^G[P]
\cong N_{P, \phi}^G[P] \cong E_\phi$. For
part (2), we use the orthogonality relation
for irreducible Brauer characters, which can
be found in Feit \cite[4.3.3]{Fei82}. Fixing
$P$, let $\cL$ be the matrix such that the
rows are indexed by the conjugacy classes
$[s]$ of $p'$-elements $s$ of $\oo{N}_G(P)$,
the columns are indexed by the irreducible
Brauer characters $\phi$ of $\FF \oo{N}_G(P)$
and the $([s], \phi)$ entry is $\hh{\phi}(s)$.
By part (1), $\cL$ is a submatrix of $\cM$
and $\cN$. Since $\cM$, $\cN$, $\cM^{-1}$,
$\cN^{-1}$ all have an upper triangular form
in the sense described in the proof of Lemma
\ref{3.1}, $\cL^{-1}$ is a submatrix of
$\cM^{-1}$ and $\cN^{-1}$. By the
orthogonality relation, $\cL^{-1}$ has
$(\phi, [s])$ entry
\begin{equation*}
\phi(s^{-1}) |[s]| / | \oo{N}(P)| = \phi(s^{-1})
  / |C_{\oo{N}(P)}(s)| \; . \qedhere
\end{equation*}
\end{proof}

We have nothing further to say about the
matrices $\cM$ and $\cM^{-1}$. Our concern,
for the rest of this section, will be with the
matrices $\cN$ and $\cN^{-1}$. We regard
$\cN$ as an analogue of the table of marks.
Our next main task is to find a general
formula for the entries of $\cN$.

\begin{lem} \label{3.3}
Let $P$ and $Q$ be $p$-subgroups of $G$. Let
$E$ be a projective $\FF \oo{N}_G(P)$-module and
let $N = \bbb_G^\ppp \Ind \bbb_{N_G(P)}^\ppp
\Inf \bbb_{\oo{N}_G(P)}^\ppp (E)$. Then
$$N[Q] = \bigoplus_{N_G(Q) g N_G(P) \subseteq G
  \: : \: Q \leq {}^g \! P} \bbb_{\oo{N}_G(Q)}^\ppp
  \Def \bbb_{N_G(Q)}^\ppp \Ind \bbb_{N_G(Q)
  \cap {}^g \! N_G(P)}^\ppp ({}^g \! E)$$
where ${}^g \! E = \bbb_{N_G(Q) \cap {}^g \!
N_G(P)}^\ppp\Con \bbb_{N_G(Q)^g \cap
N_G(P)}^g \Res \bbb_{N_G(P)}^\ppp
\Inf \bbb_{\oo{N}_G(P)}^\ppp (E)$.
\end{lem}

\begin{proof}
As a direct sum of $\FF Q$-modules,
$$N = \bigoplus_{Q g N(P) \subseteq G} N_g$$
where $N_g = \FF Q g N(P) \otimes_{\FF
N(P)} E$. We have
$$N_g \cong \bbb_Q^\ppp \Ind
  \bbb_{Q \cap {}^g \! N(P)}^\ppp \Res
  \bbb_{{}^g \! N(P)}^\ppp (\FF g N(P)
  \otimes_{\FF N(P)} E) \; .$$
So if $Q \not\leq {}^g \! N(P)$ then $N_g[Q] = 0$.
Suppose that $Q \leq {}^g \! N(P)$. Then
$$N_g = \bbb_Q^\ppp \Res \bbb_{{}^g \!
  N(P)}(\FF g N(P) \otimes_{\FF N(P)} E) \; .$$
Let $R = Q^g . P/P$. Since $R$ is a $p$-subgroup
of $\oo{N}(P)$ and $E$ is a direct summand of the
regular $\FF \oo{N}(P)$-module, $E$ has a basis
$B_g$ upon which $R$ acts fixed-point-freely.
The basis $g \otimes B_g$ of $N_g$ is
$Q$-stable. If $Q \not\leq {}^g \! P$, then $R$
is non-trivial and every $Q$-orbit of
$g \otimes B_g$ is non-singleton, hence
$N_g[Q] = 0$. On the other hand, if $Q \leq
{}^g \! P$, then $R$ is trivial and $Q$ fixes
$g \otimes B_g$, hence we can make
an identification $N_g[Q] = N_g$. Thus, we
can make an identification
$$N[Q] = \bigoplus_{QgN(P) \subseteq
  G \: : \: Q \leq {}^g \! P} N_g = \bigoplus_{N(Q)
  g N(P) \subseteq G \: : \: Q \leq {}^g \! P}
  \FF N(Q) g N(P) \otimes_{\FF N(P)} E \; .$$
As an $\FF N_G(Q)$-module fixed by $Q$,
we have
\begin{equation*}
  \FF N(Q) g N(P) \otimes_{\FF N(P)} E \cong
  \bbb_{N(Q)}^\ppp \Ind \bbb_{N(Q)
  \cap {}^g \! N(P)}^\ppp ({}^g \! E) \; . \qedhere
\end{equation*}
\end{proof}

\begin{pro} \label{3.4}
Let $(Q, [s]) \in \cE(G)$ and $(R, \psi) \in \cC(G)$.
Write $s = tQ$ with $t \in N_G(Q)$. Then
$$n_G(Q, s; R, \psi) = \sum_{g \in G \: : \:
  {}^g Q \leq R, \, {}^g t \in N_G(R)} \hh{\psi}
  ({}^g t . R) / |N_G(R)| \; .$$
\end{pro}

\begin{proof}
Abusing notation, writing ${}^g \hh{\psi}$ to
denote the conjugate by $g$ of a restriction of
an inflation of $\hh{\psi}$, we have
$$\big( \bbb_{N(Q)}^\ppp \ind \bbb_{N(Q) \cap
  {}^g \! N(R)}^\ppp ({}^g \hh{\psi}) \big) (t)
  = \sum_{y \in N(Q) \: : \: {}^y t \in
  {}^g \! N(R)} {}^g \hh{\psi} ({}^y t)
  / |N(Q) \cap {}^g N(R)| \; .$$
So, by Lemma \ref{3.3},
$$n_G(Q, s; R, \psi) = \sum_{N(Q) g N(R)
  \subseteq G, \, y \in N(Q) \: : \: Q \leq
  {}^g \! R, \, {}^y t \in {}^g \! N(R)} {}^g
  \hh{\psi}({}^y t) / |N(Q) \cap {}^g \! N(R)|$$
$$= \sum_{g \in G, \, y \in N(Q) \: : \: Q \leq
  {}^g \! R, \, {}^y t \in {}^g \! N(R)} {}^g
  \hh{\psi}({}^y t) / |N(Q)| |N(R)| \; .$$
Making the substitution $f = g^{-1}$ and
noting that ${}^g \psi({}^y t) = \psi({}^{fy} t)$,
we have
$$n_G(Q, s; R, \psi) = \sum_{f \in G, \, y \in
  N(Q) \: : \: {}^f \! Q \leq R, \, {}^{fy} t \in N(R)}
  \hh{\psi} ({}^{fy} t . R) / |N(Q)| |N(R)|$$
\begin{equation*}
  \sum_{fN(Q) \subseteq G, \, y \in N(Q) \: : \:
  {}^{fy} Q \leq R, \, {}^{fy} t \in N(R)} \hh{\psi}
  ({}^{fy} t) / |N(R)| \; . \qedhere
\end{equation*}
\end{proof}

We shall be making use of the following more
complicated formula for $n_G(Q, s; P, \phi)$.

\begin{pro} \label{3.5}
Given $(Q, [s]) \in \cE(G)$ and $(R, \psi) \in
\cC(G)$, then
$$n_G(Q, s; R, \psi) = \sum_{z \in G, \, tQ \in [s]
  \: : \: Q \leq {}^z \! R, \, t \in N_G({}^z \! R)} {}^z
  \hh{\psi} (t . {}^z \! R) \, |C_{\oo{N}_G(Q)}(s)| /
  |\oo{N}_G(Q)| |N_G(R)| \; .$$
\end{pro}

\begin{proof}
This follows from the previous proposition by
substituting $z = g^{-1}$. The
factor $1/|[s]| = |C_{\oo{N}(Q)}(s)|/|\oo{N}(Q)|$
appears because we are summing over all the
elements of $[s]$.
\end{proof}

\begin{thm} \label{3.6}
Given $(P, \phi) \in \cC(G)$ and $(Q, [s]) \in
\cE(G)$, then
$$n_G^{-1}(P, \phi; Q, s) = \sum_{vN_G(Q)
  \subseteq G, \, aP \in \oo{N}_G(P)_{p'}
  \: : \: P \leq {}^v Q, \, a \in N_G({}^v Q), \,
  a . {}^v Q \in {}^v [s]} \dozback
  \dozback \dozback \phi(a^{-1} P) \,
  \mu_a(P, {}^v Q) / | \oo{N}_G(P)| \; .$$
\end{thm}

\begin{proof}
Note that the condition $P \leq {}^v Q$ implies
that the conditions $a \in N({}^v Q)$ and
$a . {}^v Q \in {}^v [s]$ depend only on $aP$,
not on the choice of coset representative $a$. Let
$\nu_G^{-1}(P, \phi; Q, s)$ denote the right-hand
side of the asserted equality. Fixing $(P, \phi)$ and
another element $(R, \psi) \in \cE(G)$, let
$$\Delta = \sum_{(Q, [s]) \in_G \cE(G)}
  \nu_G^{-1}(P, \phi; Q, s) \, n_G(Q, s; R, \psi) \; .$$
If we can show that
$$\Delta = \left\{ \begin{array}{ll}
  1 & \mbox{\rm{if $(P, \phi) =_G (R, \psi)$,}} \\
  0 & \mbox{\rm{otherwise,}} \end{array} \right.$$
then the equality $n_G^{-1}(P, \phi; Q, s) =\
\nu_G^{-1}(P, \phi; Q, s)$ will follow. Summing over
all the elements $(Q, [s])$ of $\cE(G)$, we have
$$\Delta = \sum_{(Q, [s]) \in \cE(G)} \frac{|N(Q)|}{|G|}
  \, \nu_G^{-1}(P, \phi; Q, s) \, n_G(Q, s; R, \psi) \; .$$
Let $\Gamma = |\oo{N}(P)| |N(R)| \Delta$.
Applying Proposition \ref{3.5}, summing over all the
elements $v$ of $G$, we have
$$|G| \Gamma = \sum_{(Q, [s]), v, aP, z, tQ}
  \phi(a^{-1} P) \, \mu_a(P, {}^v Q) . {}^z \hh{\psi}
  (t . {}^z R) \, |C_{\oo{N}(Q)}(s)| / |\oo{N}(Q)|$$
with indices $(Q, [s]) \in \cE(G)$, $v \in G$,
$aP \in \oo{N}(P)_{p'}$, $z \in G$, $tQ \in [s]$
subject to the conditions
$$P \leq {}^v Q \; , \;\;\;\;\;\; a \in N({}^v Q) \; ,
  \;\;\;\;\;\; a. {}^v Q \in {}^v [s] \; , \;\;\;\;\;\;
  Q \leq {}^z \! R \; , \;\;\;\;\;\; t \in N({}^z \! R) \; .$$
Replacing $tQ$ with $u . {}^v Q \in {}^v [s]$ where
$u = {}^v t$, also replacing $z$ with $y \in G$
where $y = vz$, we have ${}^z \hh{\psi} (t . {}^z \!
R) = {}^y \hh{\psi}(u . {}^y \! R)$ and
$$|G| \Gamma = \sum_{(Q, [s]), v, aP, y, u. {}^v Q}
  \phi(a^{-1} P) \, \mu_a(P, {}^v Q) . {}^y
  \hh{\psi} (u. {}^y R) \, |C_{\oo{N}(Q)}(s)|
  / |\oo{N}(Q)|$$
where the indices are subject to
$$P \leq {}^v Q \leq {}^y \! R \; , \;\;\;\;\;\;
  a \in N({}^v Q) \; , \;\;\;\;\;\; a . {}^v Q \in
  {}^v [s] \; , \;\;\;\;\;\; u \in N({}^y \! R) \; .$$
As $(Q, [s])$ and $v$ vary, the element $({}^v Q,
{}^v [s])$ takes each value $|G|$ times. So
$$\Gamma = \sum_{(Q, [s]), aP, y, uQ} \phi(a^{-1} P)
  \, \mu_a(P, Q) \, . {}^y \hh{\psi} (u. {}^y R) \,
  |C_{\oo{N}(Q)}(s)| / |\oo{N}(Q)|$$
with indices $(Q, [s]) \in \cE(G)$, $aP \in
\oo{N}(P)_{p'}$, $y \in G$, $uQ \in [s]$ subject to
$$P \leq Q \leq {}^y \! R \; , \;\;\;\;\;\;
  a \in N(Q) \; , \;\;\;\;\;\; aQ \in [s] \; , \;\;\;\;\;\;
  u \in N({}^y \! R) \; .$$
We now replace $uQ$ with $cQ \in \oo{N}(Q)$
where $uQ = {}^{cQ} aQ$. For each value of
$uQ$, there are $|C_{\oo{N}(Q)}(s)|$ associated
values of $cQ$. We also replace $y$ with
$x \in G$ where $x = c^{-1} y$, obtaining
$$\Gamma = \sum_{(Q, [s]), aP, x, cQ}
  \phi(a^{-1} P) \, \mu_a(P, Q) .
  {}^x \hh{\psi} (a. {}^x R) / |\oo{N}(Q)|$$
with indices subject to
$$P \leq Q \leq {}^x R \; , \;\;\;\;\;\; a \in N(Q)
  \cap N({}^x R) \; , \;\;\;\;\;\; aQ \in [s] \; .$$
The term of the sum does not depend on
$cQ$, so
$$\Gamma = \sum_{(Q, [s]), aP, x} \phi(a^{-1} P)
  \, \mu_a(P, Q) . {}^x \hh{\psi} (a. {}^x R) \; .$$
The conditions $aP \in \oo{N}(P)_{p'}$ and
$a \in N(Q) \cap N({}^x R)$ imply that $aQ$
is a $p'$-element of $\oo{N}(Q)$. So we can
remove $[s]$ from the indexing, obtaining
$$\Gamma = \sum_{x, aP} \phi(a^{-1} P) .
  {}^x \hh{\psi} (a. {}^x R) \sum_Q \mu_a(P, Q)$$
with the indices $x \in G$ and $aP \in
\oo{N}(P)_{p'}$ satisfying $P \leq {}^x R$ and
$a \in N({}^x R)$, the index $Q$ running over
the subgroups $P \leq Q \leq {}^x R$ such that
$a \in N(Q)$. In other words, $Q$ runs over the
elements of the closed interval $[P, {}^x R]$
in the poset $\cS_p(G)^{\la a \ra}$. We now
apply the recurrence relation for M\"{o}bius
functions. If $P \neq_G R$ then, for each $x$
and $aP$, we have $\sum_Q \mu_a(P, Q) = 0$,
hence $\Gamma = 0$ and $\Delta = 0$.

It remains only to deal with the case where
$P =_G R$. We may assume that $P = R$
and it suffices to show that if $\phi = \psi$
then $\Delta = 1$, otherwise $\Delta = 0$.
We have $\sum_Q \mu_a(P, Q) = \mu_a(P,P)
= 1$. Using the
latest formula for $\Gamma$ and the
orthogonality relation for Brauer characters,
$$\Delta = \sum_{x \in N(P), aP \in
  \oo{N}(P)_{p'}} \phi(a^{-1} P) . {}^x \hh{\psi}
  (a . {}^x P) / |\oo{N}(P)| |N(P)|$$
\begin{equation*}
= \sum_{aP \in \oo{N}(P)_{p'}} \phi(a^{-1} P) .
  \hh{\psi}(aP) / |\oo{N}(P)| = \left\{ \begin{array}{ll}
  1 & \mbox{\rm{if $\phi = \psi$,}} \\
  0 & \mbox{\rm{otherwise.}}
  \end{array} \right. \qedhere
\end{equation*}
\end{proof}

The formula in the theorem still holds if we
divide the term by $|N(Q)|$ and sum over all
the elements $v \in G$ instead of the coset
representatives $v N(Q) \subseteq G$. Now
making the substitutions $w = v^{-1}$ and
$g = {}^w a$, the formula becomes
$$n_G^{-1}(P, \phi; Q, s) = \sum_{w \in G, \,
  g . {}^w P \in \oo{N}_G({}^w P)_{p'}
  \: : \: {}^w P \leq Q, \, b \in N_G(Q), \,
  gQ \in [s]} \dozback \dozback \dozback
  {}^w \phi(g^{-1} . {}^w P) \, \mu_g({}^w P, Q)
  / | \oo{N}_G(P)| |N_G(Q)| \; .$$
Meanwhile, summing over all the elements
$(P, \phi)$ of $\cC(G)$, we have
$$e_{Q, s}^G = \frac{1}{|G|} \sum_{(P, \phi)
  \in_G \cC(G)} |N(P)| \, n_G^{-1}(P, \phi; Q, s)
  \, [N_{P, \phi}^G] \; .$$
Combining the latest two equalities, and
noting that each element of $\cC(G)$ appears
$|G|$ times as $({}^w P, {}^w \phi)$, we
obtain the formula for $e_{Q, s}^G$ in
Theorem \ref{2.4}.

\section{The linearization map}

Letting $V \leq G$ and letting $\nu$ be the
Brauer character of a $1$-dimensional
$\FF V$-module $\FF_\nu$, then
$\bbb_G^\ppp \Ind \bbb_V^\ppp (\FF_\nu)$
is a trivial source $\FF G$-module. In this
section, we shall explicitly express the
isomorphism class $[\bbb_G^\ppp \Ind
\bbb_V^\ppp (\FF_\nu)]$ as a linear
combination of the elements of the canonical
basis of $T_\FF(G)$.

Let us make some comments on interpretation.
An {\bf $\FF^\times$-fibred $G$-set}, recall,
is defined to be an $\FF^\times$-free
$\FF^\times \times G$-set with finitely many
$\FF^\times$-orbits. The $\FF^\times$-orbits
are called the {\bf fibres}. The
{\bf $\FF$-monomial Burnside ring}, denoted
$\BB_\FF(G)$, is defined to be the $\ZZ$-module
generated by the isomorphism classes of
$\FF^\times$-fibred $G$-sets, subject to the
relation $[X] + [Y] = [X \sqcup Y]$ where $X$
and $Y$ are $\FF^\times$-fibred $G$-sets and
$[X]$ denotes the isomorphism class of $X$. We
make $B_\FF(G)$ become a unital ring with
multiplication given by $[X][Y] = [X \otimes Y]$
where $X \otimes Y$ is the set of
$F^\times$-orbits of $X \times Y$ under the
action such that $\lambda \in \FF^\times$
sends $(x, y) \in X \times Y$ to $(\lambda x,
\lambda^{-1} y)$. Let $\FF X$ denote the
$\FF G$-module, well-defined up to isomorphism,
such that there is an embedding of $\FF^\times
\times G$-sets $X \hookrightarrow \FF X$
whereby any set of representatives of the fibres
of $X$ becomes a basis for $\FF X$. Let
$$\lin_G \: : \: B_\FF(G) \rightarrow T_\FF(G)$$
be the unital ring homomorphism such that
$\lin_G[X] = [\FF X]$. We mention that $\lin_G$
is surjective and that $\lin : B_\FF \rightarrow
T_\FF$ is a morphism of biset functors. For
those two results and further discussions of
$B_\FF(G)$ and the biset functor $B_\FF$,
see \cite{Bar04} and Boltje \cite{Bol98},
\cite{Bol}. The isomorphism classes of
transitive $F^\times$-fibred $G$-sets
comprise a $\ZZ$-basis for $B_\FF(G)$.
These basis elements have the form
$[\bbb_G^\ppp \Ind \bbb_V^\ppp
(\FF_\nu^\times)]$ where $V$ and $\nu$
are as above and $\FF_\nu^\times =
\FF_\nu - \{ 0 \}$. We have $\lin_V
[\FF_\nu^\times] = [\FF_\nu]$ and
$$\lin_G [ \bbb_G^\ppp \Ind \bbb_V^\ppp
  (\FF_\nu^\times) ] = [ \bbb_G^\ppp \Ind
  \bbb_V^\ppp (\FF_\nu)] \; .$$
The formula in the next theorem can be
viewed as a formula for the matrix of
$\lin_G$ with respect to the above basis
for $B_\FF(G)$ and the canonical basis
for $T_\FF(G)$.

For a $p$-subgroup $P$ of $G$ and an
arbitrary subgroup $V$ of $G$, we write
$(P, V]_{\cS_p}$ to denote the poset of
$p$-subgroups $Q$ such that $P < Q \leq
V$. Given $g \in N_G(P) \cap V$, we write
$(P, V]_{\cS_p}^{\la g \ra}$ to denote the
subposet of $(P, V]_{\cS_p}$ consisting of
those $Q$ such that $g \in N_G(Q)$.

\begin{thm} \label{4.1}
Given $V$ and $\nu$ as above, then
$$[\bbb_G^\ppp \Ind \bbb_V^\ppp
  (\FF_\nu)] = \frac{-1 \;\;}{|V|} \sum_{(P,
  \phi) \in \cC(G), \, gP \in \oo{N}_G(P)_{p'}
  \: : \: \la P, g \ra \leq V} \!\!\!\! \!\!\!\!
  \!\!\!\! |P| \, \phi(g^{-1}
  P) \, \nu(g) \, \tt{\chi}((P, V]_{\cS_p}^{\la g
  \ra}) \, [N_{P, \phi}^G] \; .$$
\end{thm}

\begin{proof}
Let $I = [\bbb_G^\ppp \Ind \bbb_V^\ppp
(F_\nu)]$. Summing over representatives of
$G$-orbits, write
$$I = \sum_{(P, \phi) \in_G \cC(G)}
  \lambda_{P, \phi} [N_{P, \phi}^G] \; .$$
We are to evaluate the coefficients
$\lambda_{P, \phi} \in \KK$. Consider an
element $(Q, [s]) \in \cE(G)$ and let
$\hh{s} \in N_G(Q)$ such that $s =
\hh{s} Q$. We have
$$(\bbb_G^\ppp \Ind \bbb_V^\ppp
  (\FF_\nu))[Q] = \bigoplus_{fV \subseteq G
  \: : \: Q \leq {}^f V} f \otimes \FF_\nu \; .$$
Supposing that $Q \leq {}^f V$, then $s$
stabilizes $f \otimes \FF_\nu$ if and only if
$\la Q, \hh{s} \ra \leq {}^f V$, in which case,
$s$ acts on $f \otimes \FF_\nu$ as
multiplication by ${}^f \! \nu(\hh{s})$.
Therefore
$$\epsilon_{Q, s}^G(I) = \sum_{fV \subseteq
  G \: : \: \la Q, \hh{s} \ra \leq {}^f V}
  {}^f \! \nu(\hh{s}) \; .$$
It follows that
$$I = \sum_{(Q, [s]) \in_G \cE(G)}
  \epsilon_{Q, s}^G(I) \, e_{Q, s}^G =
  \frac{1}{|G|} \sum_{(Q, [s]) \in \cE(G),
  fV \subseteq G \: : \: \la Q, \hh{s} \ra
  \leq {}^f V} \!\!\!\! \!\!\!\! \!\!\!\! |N_G(Q)|
  . {}^f \! \nu(\hh{s}) \, e_{Q, s}^G \; .$$
Note that there is an arbitrary choice of
representative element $s$ of $[s]$ and
there is an arbitrary choice of lift $\hh{s}$
of $s$. Given $(Q, [s])$, then the range of
the index $fV$ depends on the choice of
$s$ but, given $(Q, [s])$ and $s$, then
the range of $fV$ does not depend on
the choice of $\hh{s}$. Applying
Theorem \ref{2.4},
$$I = \frac{1}{|G|} \sum_{(Q, [s]), fV, (P, V),
  gP} |P| . {}^f \! \nu(\hh{s}) \, \phi(g^{-1} P)
  \, \mu_g(P, Q) \, [N_{P, \phi}^G]$$
summed over $(Q, [s]) \in \cE(G)$, $fV
\subseteq G$, $(P, \phi) \in \cC(G)$,
$gP \in \oo{N}_G(P)_{p'}$ subject to
$$\la Q, \hh{s} \ra \leq {}^f V \; , \;\;\;\;\;\;\;\;
  P \leq Q \; , \;\;\;\;\;\;\;\; g \in N_G(Q) \; ,
  \;\;\;\;\;\;\;\; gQ \in [s] \; .$$
As we noted above, given $(Q, [s])$, the
choices of $s$ and $\hh{s}$ are arbitrary.
So, changing the order of summation by
giving the index $gP$ priority over the index
$fV$, we may replace $\hh{s}$ with $g$.
Bearing in mind that $(P, \phi)$ has $|G :
N_G(P)|$ conjugates, we obtain
$$\lambda_{P, \phi} = \frac{1}{|N_G(P)|}
  \sum_{gP, fV} |P| . {}^f \! \nu(g) \,
  \phi(g^{-1} P) \, \sum_Q \mu_g(P, Q)$$
summed over $gP \in \oo{N}_G(P)_{p'}$,
$fV \subseteq G$, $Q \in \cS_p(G)$
subject to the two conditions
$$P \leq Q \leq {}^f V \; , \dozspace
  g \in N_G(Q) \cap {}^f V \; .$$
Those two conditions can be rewritten as
$$\la P, g \ra \leq {}^f V \; , \dozspace
  Q \in [P, {}^f V]_{\cS_p}^{\la g \ra}$$
where $[P, {}^f V]_{\cS_p}^{\la g \ra}$
denotes the set of $p$-subgroups $Q$
of $G$ such that $P \leq Q \leq {}^f V$
and $g \in N_G(P)$. By the definition of
the M\"{o}bius function,
$$\mu_g(P, Q) = \sum_{R_{-1}, ...,
  R_{n+1}} (-1)^n$$
summed over $R_{-1}, ..., R_{n+1} \in
[P, {}^f V]_{\cS_p}^{\la g \ra}$ such that
$P = R_{-1} < ... < R_{n+1} = Q$ (allowing
the possibilities $n = -2$ and $n = -1$).
Therefore,
$$\sum_Q \mu_g(P, Q) = - \tt{\chi}
  ((P, {}^f V]_{\cS_p}^{\la g \ra}) \; .$$
We have shown that
$$\lambda_{P, \phi} = \frac{-1 \;\;}{|N_G(P)|}
  \sum_{gP \in \oo{N}_G(P)_{p'}, fV \subseteq
  G \: : \: \la P, g \ra \leq {}^f V}
  |P| \, \phi(G^{-1} P) . {}^f \! \nu(g) \,
  \tt{\chi}((P, {}^f V]_{\cS_p}^{\la g \ra}) \; .$$
Again bearing in mind that $(P, \phi)$ has
$|G : N_G(P)|$ conjugates,
$$I = \frac{-1 \;\;}{|G|} \! \sum_{(P, \phi),
  gP, fV} |P| \, \phi(G^{-1} P) . {}^f \! \nu(g)
  \, \tt{\chi}((P, {}^f V]_{\cS_p}^{\la g \ra})
  \, [N_{P, \phi}^G]$$
summed over $(P, \phi) \in \cC(G)$ and
indices $gP$ and $fV$ as before. Dividing
by $|V|$ and replacing the index $fV$ with
$f \in G$, then making the substitution
$h = f^{-1}$, we have
$$I = \frac{1 \;\;}{|G||V|} \sum_{(P, \phi),
  gP, h} |P| . {}^h \! \phi({}^h(g^{-1} P)) \,
  \nu({}^h \! g) \, \tt{\chi}(({}^h \! P,
  V]_{\cS_p}^{\la {}^h \! g \ra}) \,
  [N_{{}^h \! P, {}^h \phi}^G ]$$
summed over $(P, \phi) \in \cC(G)$,
$gP \in \oo{N}_G(P)_{p'}$, $h \in G$ such
that $\la {}^h \! P, {}^h \, g \ra \leq G$.
Rearranging the sum so that $h$ becomes
the first index, then replacing $P$, $\phi$,
$g$ with $P^h$, $\phi^h$, $g^h$, we
obtain the required equality.
\end{proof}

\begin{cor} \label{4.2}
Given $V$ and $\nu$ as above, writing
$[\bbb_G^\ppp \Ind \bbb_V^\ppp (F_\nu)]$
as a linear combination of the elements of
the canonical basis for $T_\FF(G)$ and
supposing that a given basis element
$[N_{P, \phi}^G]$ has non-zero coefficient,
then $O_p(V) \leq {}^x P \leq V$.
\end{cor}

\begin{proof}
By the latest theorem, there exist
$x \in G$ and $g \in N_G({}^x P)$ such
that $\la {}^x P, g \ra \leq V$ and
$\tt{\chi}(({}^x P, V]_{\cS_p}^{\la g \ra})
\neq 0$. For a contradiction, suppose
that $O_p(V) \not\leq {}^x P$. We may
assume that $x = 1$. Then the poset
$(P, V]_{\cS_p}^{\la g \ra}$ admits a
conical contraction $Q \mapsto Q O_p(V)
\mapsto P O_p(V)$, contradicting the
condition that the reduced Euler
characteristic is non-zero.
\end{proof}

\end{document}